\documentclass{amsart}
\usepackage{amssymb}
\usepackage{graphicx}

\theoremstyle{plain}
\newtheorem{theorem}{Theorem}[section]
\newtheorem{corollary}[theorem]{Corollary}
\newtheorem{definition}[theorem]{Definition}
\newtheorem{example}[theorem]{Example}
\newtheorem{proposition}[theorem]{Proposition}
\newtheorem{lemma}[theorem]{Lemma}
\numberwithin{equation}{section}

\begin{document}

\title[Binary and $b$-ary overpartitions]{Polynomials and algebraic curves related to certain binary and $b$-ary overpartitions}

\author{Karl Dilcher}
\address{Department of Mathematics and Statistics\\
         Dalhousie University\\
         Halifax, Nova Scotia, B3H 4R2, Canada}
\email{dilcher@mathstat.dal.ca}
\author{Larry Ericksen}
\address{1212 Forest Drive, Millville, NJ 08332-2512, USA}
\email{LE22@cornell.edu}
\thanks{Research supported in part by the Natural Sciences and Engineering
        Research Council of Canada, Grant \# 145628481}

\date{}

\setcounter{equation}{0}

\begin{abstract}
We begin by considering a sequence of polynomials in three variables whose 
coefficients count restricted binary overpartitions with certain properties. 
We then concentrate on two specific subsequences that are closely related to
the Chebyshev polynomials of both kinds, deriving combinatorial and 
algebraic properties of some special cases. We show that the zeros of these
polynomial sequences lie on certain algebraic curves, some of which we study
in greater detail. Finally, we extend part of this work to restricted 
$b$-ary overpartitions for arbitrary integers $b\geq 2$.
\end{abstract}

\maketitle

\section{Introduction}\label{sec:1}

While the basic theory of binary partitions goes back to Euler 
\cite[p.~162ff.]{Eu}, quite recently R{\o}dseth and Sellers \cite{RS} 
introduced and studied $b$-ary overpartitions for a fixed integer base 
$b\geq 2$, in analogy to ordinary overpartitions that had been introduced a 
little earlier by Corteel and Lovejoy \cite{CL}.

A {\it $b$-ary overpartition} of an integer $n\geq 1$ is a non-increasing
sequence of nonnegative integer powers of $b$ whose sum is $n$, and where the
first occurrence of a power $b$ may be overlined. We denote the number of 
$b$-ary overpartitions by $\overline{S}_b(n)$, which differs from the notation 
in \cite{RS}.

\begin{example}\label{ex:1.1}
{\rm (See \cite[p.~346]{RS}). The binary overpartitions of $n=4$ are
\[
4,\;\overline{4},\;2+2,\;\overline{2}+2,\;2+1+1,\;2+\overline{1}+1,\;
\overline{2}+1+1,\;\overline{2}+\overline{1}+1,\; 1+1+1+1,\;
\overline{1}+1+1+1.
\]
Thus $\overline{S}_2(4)=10$.}
\end{example}

As is illustrated in this example, the overlined parts form a $b$-ary
partition into distinct parts, while the non-overlined parts form an 
ordinary $b$-ary partition. We can now see that the generating function is
\begin{equation}\label{1.1}
\sum_{n=0}^\infty\overline{S}_b(n)q^n 
= \prod_{j=0}^\infty\frac{1+q^{b^j}}{1-q^{b^j}}.
\end{equation}

The concept of a $b$-ary overpartition can be restricted in different ways,
only one of which we will consider here. In analogy to the restricted $b$-ary 
partitions, such as hyperbinary representations, we restrict the number of 
times a non-overlined power of $b$ may occur in a $b$-ary overpartition; we 
denote this number by $\lambda$ and call such overpartitions
$\lambda$-{\it restricted\/}. In this case the generating function is
\begin{equation}\label{1.2}
\sum_{n=0}^\infty\overline{S}_b^\lambda(n)q^n
=\prod_{j=0}^\infty\left(1+q^{b^j}\right)
\left(1+q^{b^j}+q^{2\cdot b^j}+\dots+q^{\lambda\cdot b^j}\right),
\end{equation}
where $\overline{S}_b^\lambda(n)$ is the number of $b$-ary overpartitions of
$n$ in which each non-overlined power of $b$ may occur at most $\lambda$ times. 

\begin{example}\label{ex:1.2}
{\rm Let $b=\lambda=2$. Then \eqref{1.2} becomes
\begin{equation}\label{1.3}
1+2q+4q^2+5q^3+8q^4+10q^5+13q^6+14q^7+18q^8+21q^9+26q^{10}+\cdots
\end{equation}
Thus, in particular, $\overline{S}_2^2(4)=8$, which is consistent with 
Example~\ref{ex:1.1}, where all but the last two binary overpartitions are counted by 
$\overline{S}_2^2(4)$. The series \eqref{1.3} can also be found in \cite{ML}.
This last paper deals with the case $\lambda=b$ as well, but 
in contrast to our work it focuses on congruences of the relevant numerical
sequences.}
\end{example}

In the recent paper \cite{DE10} we defined the concept of restricted multicolor
$b$-ary partitions as a generalization of restricted $b$-ary overpartitions,
and further defined polynomial analogues of the relevant partition functions.
These polynomials then allowed us to not just count the partitions in question,
but to characterize them. We will not be concerned with this aspect of the 
theory in the present paper.

In the special case of restricted $b$-ary overpartitions with $\lambda=2$, the
polynomials introduced in \cite{DE10} specialize as follows. Let $Z=(x,y,z)$ be
a triple of variables, and $T=(r,s,t)$ a triple of positive integers. Then, in 
the notation of \cite[Def.~2.4]{DE10}, we define
\begin{equation}\label{1.4}
\sum_{n=0}^\infty\Omega_{b,T}^{(1,2)}(n;Z)q^n
=\prod_{j=0}^\infty\left(1+x^{r^j}q^{b^j}\right)
\left(1+y^{s^j}q^{b^j}+z^{t^j}q^{2\cdot b^j}\right).
\end{equation}
Comparing this with \eqref{1.2}, we immediately get, for any base $b\geq 2$,
\begin{equation}\label{1.5}
\overline{S}_b^2(n)=\Omega_{b,T}^{(1,2)}(n;1,1,1),\qquad n=0, 1, 2,\ldots,
\end{equation}
where the triple $T$ is arbitrary. 

The main purpose of this paper is to consider various aspects of the 
polynomial sequence defined by \eqref{1.4} in the special case $r=s=t=1$.
After deriving a few basic properties in Section~\ref{sec:2}, we consider two
particular subsequences in Section~\ref{sec:3}, which turn out to be closely 
related to the Chebyshev polynomials of both kinds. In the following two 
sections we then specialize the variables $x, y, z$ in two different ways, 
obtaining polynomial sequences in one, resp.\ two, variables with interesting
properties. In particular, we derive divisibility properties and combinatorial
interpretations of these polynomials. Section~\ref{sec:6} is then devoted to
the zero distribution of these and a few other related polynomial sequences.
One such curve, a particularly interesting quartic of genus 0, is studied in
greater detail in Section~\ref{sec:7}. Finally in Section~\ref{sec:8}, we
show that much of the content of Sections~\ref{sec:2}--\ref{sec:5} can be
generalized to an arbitrary integer base $b\geq 2$ with $\lambda=b$.

\section{Some basic properties}\label{sec:2}

In \cite{DE10} we derived recurrence relations for the general polynomial 
sequences that characterize all restricted multicolor $b$-ary partitions. 
In the special case $b=\lambda=2$, these recurrences take the form 
$\Omega_{2,T}^{(1,2)}(0;x,y,z)=1$, $\Omega_{2,T}^{(1,2)}(1;x,y,z)=x+y$, and
\begin{align*}
\Omega_{2,T}^{(1,2)}(2n;x,y,z)&=\Omega_{2,T}^{(1,2)}(n;x^r,y^s,z^t)
+(z+xy)\cdot \Omega_{2,T}^{(1,2)}(n-1;x^r,y^s,z^t),\\
\Omega_{2,T}^{(1,2)}(2n+1;x,y,z)&=(x+y)\cdot \Omega_{2,T}^{(1,2)}(n;x^r,y^s,z^t)
+xz\cdot \Omega_{2,T}^{(1,2)}(n-1;x^r,y^s,z^t).
\end{align*}

From this point on, we specialize further to $r=s=t=1$.
To simplify notation we set, for all $n\geq 0$,
\begin{equation}\label{2.1}
p_n(x,y,z):= \Omega_{2,T}^{(1,2)}(n;x,y,z),\qquad T=(1,1,1).
\end{equation}
Then \eqref{1.4} simplifies to the generating function
\begin{equation}\label{2.2}
\sum_{n=0}^\infty p_n(x,y,z)q^n
=\prod_{j=0}^\infty\left(1+xq^{2^j}\right)
\left(1+yq^{2^j}+zq^{2\cdot 2^j}\right),
\end{equation}
and the recurrence relations before \eqref{2.1} turn into
$p_0(x,y,z)=1$, $p_1(x,y,z)=x+y$, and for $n\geq 1$,
\begin{align}
p_{2n}(x,y,z)&=p_n(x,y,z)+(z+xy)\cdot p_{n-1}(x,y,z),\label{2.3}\\
p_{2n+1}(x,y,z)&=(x+y)\cdot p_n(x,y,z)+xz\cdot p_{n-1}(x,y,z).\label{2.4}
\end{align}
See Table~1 for the first few such polynomials, where $\Sigma$ denotes the
sums of the coefficients.

\bigskip
\begin{center}
{\renewcommand{\arraystretch}{1.1}
\begin{tabular}{|r|l|r|}
\hline
$n$ & $p_n(x,y,z)$ & $\Sigma$\\
\hline
0 & 1 & 1\\
1 & $x+y$ & 2 \\
2 & $xy+x+y+z$ & 4 \\
3 & $x^2+2xy+xz+y^2$ & 5 \\
4 & $x^2y+xy^2+xy+xz+yz+x+y+z$ & 8 \\
5 & $x^2y+x^2z+xy^2+xyz+x^2+2xy+xz+y^2+yz$ & 10 \\
6 & $x^2y^2+x^2y+xy^2+2xyz+x^2+2xy+2xz+y^2+yz+z^2$ & 13 \\
7 & $x^3+3x^2y+2x^2z+x^2yz+3xy^2+2xyz+xz^2+y^3$ & 14\\
\hline
\end{tabular}}

\medskip
{\bf Table~1}: $p_n(x,y,z)$ and sums of coefficients for $0\leq n\leq 7$.
\end{center}

\medskip
The recurrence relations \eqref{2.3}, \eqref{2.4} imply that the polynomials
$p_n(x,y,z)$ can be written in the form
\begin{equation}\label{2.5}
p_n(x,y,z) = \sum_{i,j,k\geq 0}c_n(i,j,k)\cdot x^iy^jz^k,\qquad n\geq 0.
\end{equation}
The generating function \eqref{2.2} then shows that the coefficients in
\eqref{2.5} have the following combinatorial interpretation.

\begin{proposition}\label{prop:2.1}
For any integers $n,i,j,k\geq 0$, the coefficient $c_n(i,j,k)$ in \eqref{2.5}
counts the number of $2$-restricted binary overpartitions of $n$ that have
\begin{enumerate}
\item[] $i$ different and single overlined parts,
\item[] $j$ different and single non-overlined parts, and
\item[] $k$ different pairs of non-overlined parts.
\end{enumerate}
\end{proposition}

\begin{example}\label{ex:2.1}
{\rm By \eqref{1.3} in Example~\ref{ex:1.2}, the number of 2-restricted 
binary overpartitions of $n=6$ is 13. They are, in particular,

\smallskip
$(4,2),\;(\overline{4},2),\;(4,\overline{2}),\;(\overline{4},\overline{2}),\;
(4,1,1),\;(\overline{4},1,1),\;(4,\overline{1},1),\;
(\overline{4},\overline{1},1),$

$(\overline{2},2,2),(2,2,1,1),\;(\overline{2},2,1,1),\;(2,2,\overline{1},1),\;
(\overline{2},2,\overline{1},1).$

\smallskip
\noindent
Then, for instance, the coefficient 2 in the term $2xyz$ of $p_6(x,y,z)$
(see Table~1) counts the partitions $(\overline{2},2,1,1)$ and
$(2,2,\overline{1},1)$. Similarly, the term $x^2y^2$ of $p_6(x,y,z)$ counts
the single partition $(\overline{2},2,\overline{1},1)$.}
\end{example}

The following is an obvious consequence of Proposition~\ref{prop:2.1}; we will
need it in Section~\ref{sec:4}.

\begin{corollary}\label{cor:2.2}
If we write
\[
p_n(x,1,1) = \sum_{i\geq 0} c_n(i)\cdot x^i,\qquad n\geq 0,
\]
then $c_n(i)$ counts the number of $2$-restricted binary overpartitions of $n$
with exactly $i$ overlined parts.
\end{corollary}

\section{Connections with Chebyshev polynomials}\label{sec:3}

The main objects of this section are two subsequences of the polynomial sequence
$p_n(x,y,z)$, namely 
\begin{equation}\label{3.1}
Q_n(x,y,z):=p_{2^{n+1}-2}(x,y,z),\qquad R_n(x,y,z):=p_{2^n-1}(x,y,z),
\end{equation}
for all $n\geq 0$.
Combining these definitions with the recurrence relation \eqref{2.4}, we get
as a first consequence,
\begin{equation}\label{3.1a}
R_{n+1}(x,y,z) = (x+y)\cdot R_n(x,y,z) +xz\cdot Q_{n-1}(x,y,z),\qquad n\geq 1.
\end{equation}
However, the following three-term recurrence relations will be more important.
For greater ease of notation we suppress the arguments $x, y, z$.

\begin{proposition}\label{prop:3.1}
We have $Q_0=1$, $Q_1=xy+x+y+z$, $R_0=1$, $R_1=x+y$, and for $n\geq 1$,
\begin{align}
Q_{n+1}&=(xy+x+y+z)\cdot Q_n-(x^2y+xy^2+yz)\cdot Q_{n-1},\label{3.2}\\
R_{n+1}&=(xy+x+y+z)\cdot R_n-(x^2y+xy^2+yz)\cdot R_{n-1}.\label{3.3}
\end{align}
\end{proposition}

\begin{proof}
We proceed by induction on $n$, and for further ease of notation we set
$P(n):=p_n(x,y,z)$. First, the expressions for $Q_0, Q_1, R_0, R_1$
follow from \eqref{3.1} and Table~1. Also, by \eqref{3.1} we have $Q_2=P(6)$
and $R_2=P(3)$. 
With the relevant entries in Table~1 we can now verify \eqref{3.2} and 
\eqref{3.3} for $n=1$, which is the induction beginning.

Suppose now that \eqref{3.2} and \eqref{3.3} are true up to some $n-1$ in place
of $n$; our aim is to show that they hold also for $n$, that is, as written in
\eqref{3.2} and \eqref{3.3}. By the induction hypothesis and \eqref{3.1} we have
\begin{align}
P(2^n-2)&=(xy+x+y+z)P(2^{n-1}-2)-(x^2y+xy^2+yz)P(2^{n-2}-2),\label{3.4}\\
P(2^n-1)&=(xy+x+y+z)P(2^{n-1}-1)-(x^2y+xy^2+yz)P(2^{n-2}-1).\label{3.5}
\end{align}
We multiply both sides of \eqref{3.4} by $xz$, and both sides of \eqref{3.5} by
$x+y$. Applying \eqref{2.4} three times, namely for $n$
replaced by $2^n-1$, by $2^{n-1}-1$, and by $2^{n-2}-1$, we get
\begin{equation}\label{3.6}
P(2^{n+1}-1)=(xy+x+y+z)P(2^n-1)-(x^2y+xy^2+yz)P(2^{n-1}-1),
\end{equation}
which, by \eqref{3.1}, gives \eqref{3.3}. Next, we use the induction hypothesis
again, in the form
\begin{equation}\label{3.7}
P(2^{n+1}-2)=(xy+x+y+z)P(2^n-2)-(x^2y+xy^2+yz)P(2^{n-1}-2).
\end{equation}
We multiply both sides of \eqref{3.7} by $z+xy$ and then add \eqref{3.6}.
Applying \eqref{2.3} three times and finally using the first
identity in \eqref{3.1}, we get \eqref{3.2}. This completes the proof of
Proposition~\ref{prop:3.1} by induction.
\end{proof}

With the recurrence relations \eqref{3.2} and \eqref{3.3} we can now obtain
generating functions for the two polynomial sequences.

\begin{proposition}\label{prop:3.2}
The polynomials $Q_n$ and $R_n$ satisfy the generating functions
\begin{align}
\sum_{n=0}^{\infty}Q_n(x,y,z)q^n
=\frac{1}{1-(xy+x+y+z)q+(x^2y+xy^2+yz)q^2},\label{3.8}\\
\sum_{n=0}^{\infty}R_n(x,y,z)q^n
=\frac{1-(xy+z)q}{1-(xy+x+y+z)q+(x^2y+xy^2+yz)q^2}.\label{3.9}
\end{align}
\end{proposition}

\begin{proof}
We multiply both sides of \eqref{3.8} by the denominator on the right, and
take the Cauchy product with the power series on the left. Then the constant
coefficient is $Q_0(x,y,z)=1$, while the coefficient of $x$ is zero since
$Q_1(x,y,z)=xy+x+y+z$; all other coefficients also vanish, as a consequence of
\eqref{3.2}. This proves \eqref{3.8}, and \eqref{3.9} is obtained analogously,
using \eqref{3.3}.
\end{proof}

Both Propositions~\ref{prop:3.1} and~\ref{prop:3.2} indicate that there might
be a connection with Chebyshev polynomials. This is indeed the case, as the
next result shows. We recall that the Chebyshev polynomials of the first kind,
$T_n(w)$, and of the second kind, $U_n(w)$, can be defined by the generating
functions
\begin{equation}\label{3.10}
\sum_{n=0}^{\infty}T_n(w)v^n = \frac{1-wv}{1-2wv+v^2},\qquad
\sum_{n=0}^{\infty}U_n(w)v^n = \frac{1}{1-2wv+v^2}.
\end{equation}
Using these polynomials, we can now state and prove the following identities.

\begin{proposition}\label{prop:3.3}
For all $n\geq 0$ we have
\begin{align}
Q_n(x,y,z)&=\big(x^2y+xy^2+yz\big)^{n/2}
U_n\left(\frac{xy+x+y+z}{2(x^2y+xy^2+yz)^{1/2}}\right),\label{3.11}\\
R_n(x,y,z)&=\big(x^2y+xy^2+yz\big)^{n/2}
T_n\left(\frac{xy+x+y+z}{2(x^2y+xy^2+yz)^{1/2}}\right)+\widetilde{U}_{n-1},\label{3.12} 
\end{align}
where
\[
\widetilde{U}_{n-1}:=
\frac{x+y-xy-z}{2}\cdot Q_{n-1}(x,y,z).
\]
\end{proposition}

\begin{proof}
Comparing \eqref{3.8} with the second identity in \eqref{3.10}, we see that
\[
q=\frac{v}{(x^2y+xy^2+yz)^{1/2}}\quad\hbox{and}\quad
w=\frac{xy+x+y+z}{2(x^2y+xy^2+yz)^{1/2}}.
\]
Equating coefficients of $q^n$ then gives \eqref{3.11}.

Next, with $w$ and $v$ as above, we rewrite the numerator on the right of
\eqref{3.9} as
\begin{align}
1-(xy+z)q &= (1-wv) + \frac{x+y-xy-z}{2}\cdot q \label{3.13}\\
&= (1-wv) + \frac{x+y-xy-z}{2(x^2y+xy^2+yz)^{1/2}}\cdot v.\nonumber
\end{align}
The term $1-wv$, together with the first identity in \eqref{3.10}, leads to the
first summand in \eqref{3.12}, while the second term in the last line of
\eqref{3.13} leads to $\widetilde{U}_{n-1}$ after some straightforward
manipulations.
\end{proof}

\section{A first special case: $y=z=1$}\label{sec:4}

The case $y=z=1$ is of particular interest. By a slight abuse of notation we
set
\[
Q_n(x) := Q_n(x,1,1)\qquad\hbox{and}\qquad R_n(x) := R_n(x,1,1).
\]
The first few of these polynomials are listed in Table~2. With $y=z=1$ we get
$x+y-xy-z=0$, so that $\widetilde{U}_{n-1}=0$ and 
Proposition~\ref{prop:3.3} simplifies as follows.

\begin{corollary}\label{cor:4.1}
For all $n\geq 0$ we have
\begin{align}
Q_n(x) &= \big(1+x+x^2\big)^{n/2}\cdot 
U_n\left(\frac{1+x}{(1+x+x^2)^{1/2}}\right),\label{4.1}\\
R_n(x) &= \big(1+x+x^2\big)^{n/2}\cdot 
T_n\left(\frac{1+x}{(1+x+x^2)^{1/2}}\right).\label{4.2}
\end{align}
\end{corollary}

These identities show that most properties and identities satisfied by the
Chebyshev polynomials will carry over to the polynomials $Q_n(x)$ and 
$R_n(x)$. For instance, factors of the Chebyshev polynomials (see, e.g.,
\cite[p.~227 ff]{Ri}) lead to corresponding factors of the polynomials 
$Q_n(x)$, $R_n(x)$. More will be stated in the following corollary.

\bigskip
\begin{center}
{\renewcommand{\arraystretch}{1.1}
\begin{tabular}{|r|l|l|}
\hline
$n$ & $Q_n(x)$ & $R_n(x)$ \\
\hline
0 & 1 & 1 \\
1 & $2x+2$ & $x+1$ \\
2 & $3x^2+7x+3$ & $x^2+3x+1$ \\
3 & $4x^3+16x^2+16x+4$ & $x^3+6x^2+6x+1$ \\
4 & $5x^4+30x^3+51x^2+30x+5$ & $x^4+10x^3+19x^2+10x+1$ \\
5 & $6x^5+50x^4+126x^3+126x^2+50x+6$ & $x^5+15x^4+45x^3+45x^2+15x+1$ \\
\hline
\end{tabular}}

\medskip
{\bf Table~2}: $Q_n(x)$ and $R_n(x)$ for $0\leq n\leq 5$.
\end{center}

\bigskip
\begin{corollary}\label{cor:4.2}
Let $n\geq 1$.
\begin{enumerate}
\item[(a)] $Q_n(x)$ and $R_n(x)$ are palindromic polynomials of degree $n$.
\item[(b)] Their zeros are real and negative, and with the exception of $x=-1$
they appear in pairs whose product is $1$ and whose sum can be arbitrarily large
as $n$ grows.
\item[(c)] $Q_{n-1}(x)$ is a divisibility sequence: if $m|n$, then
$Q_{m-1}(x)|Q_{n-1}(x)$.
\item[(d)] The sums of the coefficients are
\[
Q_n(1)=\tfrac{1}{2}\bigl(3^{n+1}-1\bigr),\qquad
R_n(1)=\tfrac{1}{2}\bigl(3^n+1\bigr),
\]
and in particular, $R_{n+1}(1)=Q_n(1)+1$.
\end{enumerate}
\end{corollary}

The identities in part (d) have also been obtained by Ma and Lu \cite{ML} as 
their Corollary~6 and Theorem~5, respectively. 

\begin{proof}[Proof of Corollary~\ref{cor:4.2}]
(a) With \eqref{4.1} and \eqref{4.2} it is easy to see that 
$x^nQ_n(1/x)=Q_n(x)$, and similarly for $R_n(x)$. The degree statement
follows from Proposition~\ref{prop:3.1} with $y=z=1$.

(b) With the arguments of $U_n$ and $T_n$ in Corollary~\ref{cor:4.1} in mind,
we set $w=(1+x)/(1+x+x^2)^{1/2}$. This can be rewritten as
\begin{equation}\label{4.3}
x^2 + \frac{2-w^2}{1-w^2}\cdot x + 1 = 0.
\end{equation}
It is known that the zeros of $U_n(w)$ and $T_n(w)$ lie strictly between $-1$
and 1, and so we consider $0<w^2<1$. But then it is easy to see that the
discriminant of the quadratic in \eqref{4.3} is positive, and thus for each
pair of zeros $\pm w$ of $U_n(w)$ or $T_n(w)$ there is a pair of zeros of
$Q_n(x)$ or $R_n(x)$ whose product is 1 and whose sum is $-(2-w^2)/(1-w^2)$;
this follows from the fact that the quadratic in \eqref{4.3} is itself
palindromic and is monic. Since it is known that in both cases $w$ can be 
arbitrarily close to $\pm 1$ if $n$ is sufficiently large, the sum of the
zeros of $Q_n(x)$ or $R_n(x)$ can be arbitrarily large negative, as claimed.

(c) This follows from the corresponding property of the Chebyshev polynomials
$U_n(w)$; see, e.g., \cite[p.~232]{Ri}.

(d) By \eqref{3.2} with $x=y=z=1$ we have $Q_0(1)=1$, $Q_1(1)=4$, and
for $n\geq 1$,
$Q_{n+1}(1)=4Q_n(1)-3Q_{n-1}(1)$. It is now easy to verify that the sequence
$\tfrac{1}{2}(3^{n+1}-1)$ also satisfies this recurrence relation with the same
initial conditions; hence the two sequences are identical. The proof 
for $R_n(1)$ is analogous.
\end{proof}

The identities \eqref{4.1} and \eqref{4.2} point to a possible connection
between the polynomials $Q_n(x), R_n(x)$ and the trinomial coefficients or the
trinomial triangle. The $n$th row of the {\it trinomial triangle\/} consists
of the coefficients of the polynomial $(1+x+x^2)^n$; see the entry A027907 in
\cite{OEIS}.

\begin{proposition}\label{prop:4.3} For any integer $n\geq 1$ we have
\begin{equation}\label{4.4}
x\cdot Q_{n-1}(x^2) + R_n(x^2) = (1+x+x^2)^n.
\end{equation}
In other words, the coefficients of $R_n(x)$ are the even-index entries of the
$n$th row of the trinomial triangle, while the coefficients of $Q_{n-1}(x)$
are the odd-index entries.
\end{proposition}

The even- and odd-index  entries of the rows of the trinomial triangle are
listed in \cite{OEIS} as A056241 and A123934, respectively. We were led to
Proposition~\ref{prop:4.3} through these entries. For a reformulation of
Proposition~\ref{prop:4.3}, see \eqref{4.10} below.

\begin{proof}[Proof of Proposition~\ref{prop:4.3}]
We use the defining identities
\[
\sin{\theta}\cdot U_{n-1}(\cos{\theta})=\sin(n\theta),\qquad
T_n(\cos{\theta})=\cos(n\theta).
\]
Multiplying both sides of the left identity by $i$, then adding both and using
$2i\sin{\theta}=e^{i\theta}-e^{-i\theta},\,
2\cos{\theta}=e^{i\theta}+e^{-i\theta}$, and $w:=e^{i\theta}$, we get
\begin{equation}\label{4.5}
\frac{w-w^{-1}}{2}U_{n-1}\left(\frac{w+w^{-1}}{2}\right)
+T_n\left(\frac{w+w^{-1}}{2}\right) = w^n,\qquad n\geq 1.
\end{equation}
We now set
\[
w = \sqrt{\frac{1+x+x^2}{1-x+x^2}} = \frac{1+x+x^2}{\sqrt{1+x^2+x^4}}.
\]
Then after some straightforward manipulations, which involves the 
factorization $1+x^2+x^4=(1-x+x^2)(1+x+x^2)$, we get
\begin{equation}\label{4.6}
\frac{w+w^{-1}}{2} = \frac{1+x^2}{(1+x^2+x^4)^{1/2}}.
\end{equation}
Similarly, we obtain
\begin{equation}\label{4.7}
\frac{w-w^{-1}}{2} = \frac{x}{(1+x^2+x^4)^{1/2}}.
\end{equation}
Substituting \eqref{4.6} and \eqref{4.7} into \eqref{4.5}, we get
\begin{equation}\label{4.8}
\frac{x}{(1+x^2+x^4)^{1/2}}U_{n-1}\left(\frac{w+w^{-1}}{2}\right)+
T_n\left(\frac{w+w^{-1}}{2}\right)
= \frac{(1+x+x^2)^n}{(1+x^2+x^4)^{n/2}}.
\end{equation}
Finally, multiplying both sides of \eqref{4.8} by $(1+x^2+x^4)^{n/2}$ and
using \eqref{4.1} and \eqref{4.2}, we get the desired identity \eqref{4.4}.
\end{proof}

For our next result we use a notation for the trinomial coefficient that can
be found in \cite[p.~78]{Co}: For an integer $n\geq 0$ we write 
\begin{equation}\label{4.9}
\left(1+x+x^2\right)^n = \sum_{j=0}^n\binom{n,3}{j}x^j.
\end{equation}
We can now state and prove the following interpretation of trinomial
coefficients in terms of binary overpartitions.

\begin{proposition}\label{prop:4.4}
Let $n\geq 1$ be an integer. Then
\begin{enumerate}
\item[(a)]\;$\binom{n,3}{2j}$ is the number of binary $2$-restricted 
overpartitions of $2^n-1$ with exactly $j$ overlined parts, $j=0,1,\ldots,n$;
\item[(b)]\;$\binom{n,3}{2j+1}$ is the number of binary $2$-restricted 
overpartitions of $2^n-2$ with exactly $j$ overlined parts, $j=0,1,\ldots,n-1$.
\end{enumerate}
\end{proposition}

\begin{proof}
If we set 
\[
R_n(x)=\sum_{j=0}^n r_n(j)x^j,\qquad Q_n(x)=\sum_{j=0}^n q_n(j)x^j,
\]
then Proposition~\ref{prop:4.3} can be stated as
\begin{equation}\label{4.10}
r_n(j)=\binom{n,3}{2j},\qquad q_{n-1}(j)=\binom{n,3}{2j+1},
\end{equation}
valid for $j=0,1,\ldots,n$, resp.\ for $j=0,1,\ldots,n-1$. The statements of
the proposition now follow from Corollary~\ref{cor:2.2} combined with
\eqref{3.1} for $y=z=1$.
\end{proof}

\begin{example}\label{ex:4.1}
{\rm Let $n=3$. The 2-restricted binary overpartitions of $2^3-1=7$
are

\smallskip
$(4,2,1), (\overline{4},2,1), (4,\overline{2},1), (4,2,\overline{1}),
(\overline{4},2,\overline{1}), (\overline{4},\overline{2},1),
(4,\overline{2},\overline{1}),(\overline{4},\overline{2},\overline{1}),$

$(4,\overline{1},1,1), (\overline{4},\overline{1},1,{1}), (\overline{2},2,2,1),
(\overline{2},{2},2,\overline{1}), ({2},2,\overline{1},1,1),
(\overline{2},2,\overline{1},1,1).$

\smallskip
\noindent
There are 14 of them, consistent with \eqref{1.3} in Example~\ref{ex:1.2}. 
We see that the
numbers of these partitions with 0, 1, 2, and 3 overlined parts are 1, 6, 6, and
1, respectively, which agrees with Proposition~\ref{prop:4.4}(a) and
\begin{equation}\label{4.11}
\left(1+x+x^2\right)^3  = \sum_{j=0}^6\binom{3,3}{j}x^j 
= 1+3x+6x^2+7x^3+6x^4+3x^5+x^6.
\end{equation}
Similarly, Example~\ref{ex:2.1} shows the 2-restricted binary
overpartitions $2^3-2=6$, of which there are 13. We see that the numbers of
those with 0, 1, and 2 overlined parts are 3, 7, and 3, respectively. This is
consistent with Proposition~\ref{prop:4.4}(b) and again with \eqref{4.11}.}
\end{example}

\medskip
\noindent
{\bf Remarks.} (1) It is known that for $0\leq k\leq n-1$ the zeros of $U_n(w)$
and of $T_n(w)$ are $w_k=\cos\big(\pi(k+1)/(n+1)\big)$ and
$w'_k=\cos\big(\pi(2k+1)/2n\big)$, respectively. Therefore, by solving 
\eqref{4.3} for $x$, one can easily obtain explicit expressions for the zeros
of $Q_n(x)$ and $R_n(x)$.

(2) While $R_n(x)$ is not a divisibility sequence, a weaker property still 
holds; see, e.g., \cite{Ri} for the corresponding Chebyshev analogue.

\section{A second special case: $x=y=z$}\label{sec:9}

If we set $y$ and $z$ equal to $x$ instead of 1, we get another pair of 
single-variable polynomial sequences with some interesting properties. To
distinguish the notation from that in Section~\ref{sec:4}, we set
\begin{equation}\label{9.1}
Q_n(Z):=Q_n(z,z,z)\qquad\hbox{and}\qquad
R_n(Z):=R_n(z,z,z).
\end{equation}
From Proposition~\ref{prop:3.1} we immediately get the following recurrence
relations.

\begin{corollary}\label{cor:9.1}
We have $Q_0(Z)=1$, $Q_1(Z)=z^2+3z$, 
$R_0(Z)=1$, $R_1(Z)=2z$, and for $n\geq 1$,
\begin{align}
Q_{n+1}(Z)&=(z^2+3z)\cdot Q_n(Z)
-(2z^3+z^2)\cdot Q_{n-1}(Z),\label{9.2}\\
R_{n+1}(Z)&=(z^2+3z)\cdot R_n(Z)
-(2z^3+z^2)\cdot R_{n-1}(Z).\label{9.3}
\end{align}
\end{corollary}

Using these recurrence relations, we can compute the first few terms of both
polynomial sequences, as shown in Table~3.

\bigskip
\begin{center}
{\renewcommand{\arraystretch}{1.1}
\begin{tabular}{|r|l|l|}
\hline
$n$ & $Q_n(Z)$ & $R_n(Z)$ \\
\hline
0 & 1 & 1 \\
1 & $z^2+3z$ & $2z$ \\
2 & $z^4+4z^3+8z^2$ & $5z^2$ \\
3 & $z^6+5z^5+13z^4+21z^3$ & $z^4+13z^3$ \\
4 & $z^8+6z^7+19z^6+40z^5+55z^4$ & $z^6+6z^5+34z^4$ \\
5 & $z^{10}+7z^9+26z^8+66z^7+120z^6+144z^5$ & $z^8+7z^7+25z^6+89z^5$ \\
\hline
\end{tabular}}

\medskip
{\bf Table~3}: $Q_n(Z)$ and $R_n(Z)$ for $0\leq n\leq 5$.
\end{center}

\bigskip
In order to prove some of the properties of the sequences in \eqref{9.1} that
are obvious from Table~3, we first define
\begin{equation}\label{9.4}
\widetilde{Q}_n(z):=z^{-n}Q_n(Z)\qquad\hbox{and}\qquad
\widetilde{R}_n(z):=z^{-n}R_n(Z).
\end{equation}
Then Corollary~\ref{cor:9.1} implies that $\widetilde{Q}_0(z)=1$, 
$\widetilde{Q}_1(z)=z+3$, $\widetilde{R}_0(z)=1$, $\widetilde{R}_1(z)=2$, 
and for $n\geq 1$,
\begin{align}
\widetilde{Q}_{n+1}(z)&=(z+3)\cdot\widetilde{Q}_n(z)
-(2z+1)\cdot\widetilde{Q}_{n-1}(z),\label{9.5}\\
\widetilde{R}_{n+1}(z)&=(z+3)\cdot\widetilde{R}_n(z)
-(2z+1)\cdot\widetilde{R}_{n-1}(z).\label{9.6}
\end{align}
We now state and prove the following properties of the sequences defined in
\eqref{9.1}.

\begin{lemma}\label{lem:9.2}
For each $n\geq 1$, the polynomial $Q_n(Z)$ has the following 
properties:
\begin{enumerate}
\item[(a)] It is monic of degree $2n$, with lowest term having degree $n$;
\item[(b)] the coefficient of $z^{2n-1}$ is $n+2$;
\item[(c)] the coefficient of $z^n$ is the Fibonacci number $F_{2n+2}$.
\end{enumerate}
For each $n\geq 3$, the polynomial $R_n(Z)$ has the following 
properties:
\begin{enumerate}
\item[(d)] It is monic of degree $2n-2$, with lowest term having degree $n$;
\item[(e)] for $n\geq 4$, the coefficient of $z^{2n-3}$ is $n+2$;
\item[(f)] the coefficient of $z^n$ is $F_{2n+1}$.
\end{enumerate}
\end{lemma}

\begin{proof}
Using induction with \eqref{9.5}, we see that $\widetilde{Q}_n(z)$ is monic of
degree $n$. Again with \eqref{9.5}, we see that $\widetilde{Q}_n(0)$ satisfies
the well-known recurrence relation for the even-index Fibonacci numbers; see,
e.g., \cite[A001906]{OEIS}. Next, if we write 
$\widetilde{Q}_n(z)=z^n+a_{n-1}^{(n)}z^{n-1}+\cdots$ and substitute it into
\eqref{9.5}, then upon equating coefficients of like powers of $z$, we get
$a_{n}^{(n+1)}=a_{n-1}^{(n)}+1$. With the initial condition $a_{0}^{(1)}=3$,
this gives $a_{n-1}^{(n)}=n+2$ for $n\geq 3$. All this, combined with the first
identity in \eqref{9.5}, proves parts (a)--(c) of the lemma. Parts (d)--(f)
can be obtained in a similar way by using \eqref{9.6} and its initial
conditions.
\end{proof}

\noindent
{\bf Remarks.} (1) With some further effort it would be possible to determine
coefficients other than those in Lemma~\ref{lem:9.2}. For instance, if we set
\begin{equation}\label{9.7}
\widetilde{Q}_n(z)=\sum_{j=0}^{n}a_{j}^{(n)}z^j,\qquad
\widetilde{R}_n(z)=\sum_{j=0}^{n-2}b_{j}^{(n)}z^j\quad(n\geq 2),
\end{equation}
then the sequence $(a_{1}^{(n)})_{n\geq 1}=(1, 4, 13, 40, 120,\ldots)$ is 
listed as A238846 in \cite{OEIS}, and the sequence 
$(a_{n-2}^{(n)})_{n\geq 2}=(8, 13, 19, 26,\ldots)$ is determined by
$a_{n-2}^{(n)}=(n^2+5n+2)/2$; see \cite[A034856]{OEIS}. 

(2) Similarly, the 
sequences $(b_{1}^{(n)})_{n\geq 3}=(1, 6, 25, 90, 300,\ldots)$ and
$(b_{n-2}^{(n)})_{n\geq 5}=(25, 33, 42, 52, 63,\ldots)$ are A001871 and
A055998, respectively, in \cite{OEIS}, with $b_{n-2}^{(n)}=n(n+5)/2$ for
$n\geq 5$. 
Further coefficients and their properties could be obtained by manipulating
the generating functions \eqref{3.8} and \eqref{3.9} with $x=y=z$.

\medskip
We are now ready to give combinatorial interpretations for the polynomials
considered in this section. To do so, we need the following notation.

\begin{definition}\label{def:9.3}
{\rm Let $\beta$ be any binary overpartition, restricted or not, of some 
positive integer. We define $S(\beta)$ to be the sum of the numbers of
\begin{equation}\label{9.8}
\begin{cases}
\hbox{the different and single overlined parts of $\beta$, and}\\
\hbox{the different and single non-overlined parts of $\beta$, and}\\
\hbox{the different pairs of non-overlined parts of $\beta$.}
\end{cases}
\end{equation}}
\end{definition}

We can now obtain the desired interpretation by considering the connections 
between the polynomial sequences introduced in this section and the
polynomials $p_n(x,y,z)$ of Section~\ref{sec:2}, via the identities in 
\eqref{9.1} and in \eqref{3.1}. Setting $x=y=z$ and replacing $n$ by
$2^{n+1}-2$ in \eqref{2.5}, we first obtain the following result from
Proposition~\ref{prop:2.1}.

\begin{proposition}\label{prop:9.4}
For $n\geq 1$ and $0\leq j\leq n$, the coefficient $a_j^{(n)}$, as defined in 
\eqref{9.7}, counts the number of $2$-restricted binary overpartitions $\beta$ 
of $2^{n+1}-2$ with $S(\beta)=n+j$.
\end{proposition}

\begin{example}\label{ex:5.1} 
{\rm We take $n=2$ and consider the thirteen 2-restricted binary
overpartitions of 6, as given in Example~\ref{ex:2.1}. We list them along with 
the sums corresponding to \eqref{9.8}:
\begin{center}
{\renewcommand{\arraystretch}{1.1}
\begin{tabular}{ll|ll|ll}
$(4,2)$: & $0+2+0=2$ & $(\overline{4},1,1)$: & $1+0+1=2$ & $(\overline{2},2,1,1)$: & $1+1+1=3$ \\
$(\overline{4},2)$: & $1+1+0=2$ & $(4,\overline{1},1)$: & $1+2+0=3$ & $(2,2,\overline{1},1)$: & $1+1+1=3$ \\
$(4,\overline{2})$: & $1+1+0=2$ & $(\overline{4},\overline{1},1)$: & $2+1+0=3$ & $(\overline{2},2,\overline{1},1)$: & $2+2+0=4$ \\
$(\overline{4},\overline{2})$: & $2+0+0=2$ & $(\overline{2},2,2)$: & $1+0+1=2$ & & \\
$(4,1,1)$: & $0+1+1=2$ & $(2,2,1,1)$: & $0+0+2=2$ & & \\
\end{tabular}}
\end{center}
We count 1, 4, and 8 binary overpartitions $\beta$ with $S(\beta)=4, 3$, and 2,
respectively. This corresponds to the polynomial 
$Q_2(Z) = z^4+4z^3+8z^2$, consistent with Proposition~\ref{prop:9.4}.}
\end{example}

This last example can be seen as a special case of the following corollary,
which is an easy consequence of Lemma~\ref{lem:9.2} and
Proposition~\ref{prop:9.4}.

\begin{corollary}\label{cor:9.5}
$(a)$ For $n\geq 1$, the $2$-restricted binary overpartitions $\beta$ of 
$2^{n+1}-2$ satisfy $n\leq S(\beta)\leq 2n$.
$(b)$ There are $1, n+2$, and $F_{2n+2}$ such overpartitions $\beta$ with
$S(\beta)=2n, 2n-1$, and $n$, respectively.
\end{corollary}

In analogy to Proposition~\ref{prop:9.4}, we obtain the following result by
substituting $n$ with $2^n-1$ in \eqref{2.5}.

\begin{proposition}\label{prop:9.6}
For $n\geq 3$ and $0\leq j\leq n-2$, the coefficient $b_j^{(n)}$, as defined in
\eqref{9.7}, counts the number of $2$-restricted binary overpartitions $\beta$
of $2^n-1$ with $S(\beta)=n+j$.
\end{proposition}

\begin{example}\label{ex:5.2}
{\rm We take $n=3$ and consider the fourteen 2-restricted binary
overpartitions of 7, as given in Example~\ref{ex:4.1}. Rather than listing
them all, we just note that only for $\beta=(\overline{2},2,\overline{1},1,1)$
we have $S(\beta)=2+1+1=4$, while $S(\beta)=3$ for all the others. This
corresponds to $R_3(Z) = z^4+13z^3$, consistent with 
Proposition~\ref{prop:9.6}.}
\end{example}

As an easy consequence of Proposition~\ref{prop:9.6}, together with
Lemma~\ref{lem:9.2}, we state the following analogue of Corollary~\ref{cor:9.5}.

\begin{corollary}\label{cor:9.7}
$(a)$ For $n\geq 2$, the $2$-restricted binary overpartitions $\beta$ of 
$2^n-1$ satisfy $n\leq S(\beta)\leq 2n-2$.
$(b)$ For $n\geq 4$, there are $1, n+2$, and $F_{2n+1}$ such overpartitions 
$\beta$ with $S(\beta)=2n-2, 2n-3$, and $n$, respectively.
\end{corollary}

Since the polynomials $\widetilde{Q}_n(z)$ and $\widetilde{R}_n(z)$ are
single-variable polynomials, it is of interest to explore their zero 
distribution, as we are doing in other sections as well.

\begin{proposition}\label{prop:9.8}
For each $n\geq 1$, all zeros of $\widetilde{Q}_n(z)$ lie on the circle with
radius $5/2$, centered at $-1/2$. Furthermore, the real parts of the zeros 
 are strictly less than $1$, and they are dense on this section of the circle.
\end{proposition}

\begin{proof}
By \eqref{9.1}, \eqref{9.4}, and \eqref{3.11} we have
\[
\widetilde{Q}_n(z)=(2z+1)^{n/2}\cdot U_n\left(\frac{z+3}{2(2z+1)^{1/2}}\right)
\]
or, shifted by $1/2$,
\begin{equation}\label{9.9}
\widetilde{Q}_n(z-\tfrac{1}{2})
=(2z)^{n/2}\cdot U_n\left(\frac{z+\tfrac{5}{2}}{2(2z)^{1/2}}\right).
\end{equation}
It is a well-known fact that all the zeros of $U_n(z)$ are real and lie in the
interval $(-1,1)$. Let $r$ be any such zero. Then by \eqref{9.9}, a zero of
$\widetilde{Q}_n(z-\tfrac{1}{2})$ has to satisfy 
$z+5/2=2r(2z)^{1/2}$, or
\begin{equation}\label{9.10}
z^2 + \left(5-8r^2\right)z+\frac{25}{4} = 0.
\end{equation}
It is easy to verify that for $r^2<1$ the discriminant of this quadratic is
negative, which means that the equation \eqref{9.10} has a pair of complex
conjugate zeros with product $25/4$. Hence their modulus is $5/2$, which 
proves the first statement of the proposition, keeping the shift by $1/2$ in
mind.

Solving the quadratic in \eqref{9.10} for $z$, we get
\begin{equation}\label{9.11}
z = \frac{1}{2}\left(8r^2-5\right)\pm i\cdot 2r\sqrt{5-4r^2},\qquad -1<r<1.
\end{equation}
The real part of this expression is obviously increasing with $r^2$, and the
limit as $r^2\rightarrow 1$ is $z=\frac{3}{2}\pm 2i$, which proves the second
statement. Finally, since the zeros of all the $U_n(x)$ are dense in $(-1,1)$,
the resulting $z$-values are also dense on the section of the circle
specified in \eqref{9.11}. This completes the proof.
\end{proof}

\noindent
{\bf Remarks.} (1) Since the zeros of the Chebyshev polynomials $U_n(x)$ are
known to be $r_j=\cos(j\pi/(n+1))$, $j=1,2,\ldots,n$, the zeros of 
$\widetilde{Q}_n(z)$ can also be given explicitly, via \eqref{9.9} and
\eqref{9.11}.

(2) Due to the extra term $\widetilde{U}_{n-1}$ in \eqref{3.12}, the zero
distribution of the polynomials $\widetilde{R}_n(z)$ is less straightforward
than that of $\widetilde{Q}_n(z)$. However, computations suggest that with
increasing $n$, the zeros of $\widetilde{R}_n(z)$ approach the circle given
in Proposition~\ref{prop:9.8}. We did not pursue this further since it is not
central to the current paper.

\section{A third special case: $x=1$}\label{sec:5}

In this section we will only be dealing with the polynomial sequence 
$Q_n(x,y,z)$, as defined in \eqref{3.1}. With $x=1$, the identity \eqref{3.11}
simplifies to 
\begin{equation}\label{5.1}
Q_n(1,y,z)=\big(y+y^2+yz\big)^{n/2}
U_n\left(\frac{2y+z+1}{2(y+y^2+yz)^{1/2}}\right).
\end{equation}
Thanks to the special nature of Chebyshev polynomials, this expression 
simplifies further, as follows.

\begin{proposition}\label{prop:5.1}
For any $n\geq 0$ we have
\begin{equation}\label{5.2}
Q_n(1,y,z)=\frac{1}{z+1}\left((y+z+1)^{n+1}-y^{n+1}\right).
\end{equation}
\end{proposition}

\begin{proof}
We use the well-known explicit expression
\begin{equation}\label{5.3}
U_n(w)=\frac{1}{2\sqrt{w^2-1}}\left(\left(w+\sqrt{w^2-1}\right)^{n+1}
-\left(w-\sqrt{w^2-1}\right)^{n+1}\right)
\end{equation}
(see, e.g., \cite[p.~10]{Ri}). We now set
\[
w:=\frac{2y+z+1}{2r},\qquad r:=\left(y+y^2+yz\right)^{1/2},
\]
so that
\[
w^2-1=\frac{(2y+z+1)^2}{4(y+y^2+yz)}-1 = \frac{(z+1)^2}{4(y+y^2+yz)},
\]
where the second identity is easy to verify. We then get
\begin{equation}\label{5.4}
\sqrt{w^2-1} = \frac{z+1}{2r}.
\end{equation}
Finally, combining \eqref{5.4} with \eqref{5.3} and \eqref{5.1}, we have
\[
Q_n(1,y,z) = r^n\cdot\frac{r}{z+1}
\left(\left(\frac{2y+z+1}{2r}+\frac{z+1}{2r}\right)^{n+1}
-\left(\frac{2y+z+1}{2r}-\frac{z+1}{2r}\right)^{n+1}\right).
\]
Some straightforward simplification now leads to \eqref{5.2}.
\end{proof}

Before continuing, we note that Proposition~\ref{prop:2.1} gives the
following interpretation of the polynomials $Q_n(1,y,z)$.

\begin{corollary}\label{cor:5.2}
If we write
\begin{equation}\label{5.5}
Q_n(1,y,z) = \sum_{j,k\geq 0}c_n(j,k)\cdot y^jz^k,\qquad n\geq 0,
\end{equation}
then $c_n(j,k)$ counts the number of $2$-restricted binary overpartitions of 
$2^{n+1}-2$ with $j$ different and single non-overlined parts and $k$ different
pairs of non-overlined parts.
\end{corollary}

We illustrate Corollary~\ref{cor:5.2} with an example for $n=2$.

\begin{example}\label{ex:6.1}
{\rm Noting that $2^{2+1}-2=6$, Table~1 with $x=1$ gives
\begin{equation}\label{5.6}
Q_2(1,y,z)=p_6(1,y,z)=3y^2+3y+3yz+z^2+2z+1.
\end{equation}
The thirteen 2-restricted binary overpartitions of 6 can be found in 
Example~\ref{ex:2.1}. Corollary~\ref{cor:5.2} now counts the following subsets:
 
\begin{center}
{\renewcommand{\arraystretch}{1.1}
\begin{tabular}{ll|ll}
$c_2(2,0)=3$: & $(4,2),\;(4,\overline{1},1),\;(\overline{2},2,\overline{1},1)$,
& $c_2(0,2)=1$: & $(2,2,1,1)$, \\
$c_2(1,0)=3$: & $(\overline{4},2),\;(4,\overline{2}),\;
(\overline{4},\overline{1},1)$, 
& $c_2(0,1)=2$: & $(\overline{4},1,1),\;(\overline{2},2,2)$, \\
$c_2(1,1)=3$: & $(4,1,1),\;(\overline{2},2,1,1),\;(2,2,\overline{1},1)$, 
& $c_2(0,0)=1$: & $(\overline{4},\overline{2})$.
\end{tabular}}
\end{center}}
\end{example}

We now consider two specific cases of $Q_n(1,y,z)$ that are of particular
interest, namely the polynomial sequences $Q_n(1,z,z)$ and $Q_n(1,z,z^2)$.
The first few polynomials in each sequence are listed in Tables~4 and~5 later
in this section.
We begin with an easy consequence of Corollary~\ref{cor:5.2}.

\begin{corollary}\label{cor:5.4}
$(a)$ The coefficient of $z^\mu$ in $Q_n(1,z,z)$ is the number of $2$-restricted
binary overpartitions of $2^{n+1}-2$ with $\mu$ distinct non-overlined parts.

$(b)$ The coefficient of $z^\mu$ in $Q_n(1,z,z^2)$ is the number of $2$-restricted
binary overpartitions of $2^{n+1}-2$ with a total of $\mu$ non-overlined parts.
\end{corollary}

\begin{proof}(a) This follows from Corollary~\ref{cor:5.2} with $y=z$, so that
the exponent of $z$ is $\mu=j+k$.
(b) In this case, by \eqref{5.5} the exponent of $z$ in $Q_n(1,z,z^2)$ is
$\mu=j+2k$, and the result follows again from Corollary~\ref{cor:5.2}.
\end{proof}

\begin{example}\label{ex:6.2}
{\rm (a) By \eqref{5.6} we have $Q_2(1,z,z)=7z^2+5z+1$, and
accordingly we have 7 overpartitions with two distinct non-overlined parts, 
5 with
only one  distinct non-overlined part, and 1 with none. All this is consistent
with Example~\ref{ex:6.1}.

(b) In this case, \eqref{5.6} gives the polynomial
\begin{equation}\label{5.10}
Q_2(1,z,z^2)=z^4+3z^3+5z^2+3z+1. 
\end{equation}
Thus, for instance, we have 5 overpartitions with exactly two non-overlined 
parts, namely
\[
(4,2),\;(4,\overline{1},1),\;(\overline{2},2,\overline{1},1),\;
(\overline{4},1,1),\;(\overline{2},2,2).
\]}
\end{example}

We see from \eqref{5.10} that $Q_2(1,z,z^2)$ is a palindromic (or 
self-reciprocal) polynomial. This is in fact always true:

\begin{corollary}\label{cor:5.5}
$Q_n(1,z,z^2)$ is a monic and self-reciprocal polynomial of degree $2n$ with
integer coefficients whose sum is $\frac{1}{2}(3^n-1)$.
\end{corollary}

\begin{proof}
From \eqref{5.2} we immediately get
\begin{equation}\label{5.11}
Q_n(1,z,z^2)=\frac{(z^2+z+1)^{n+1}-z^{n+1}}{z^2+1}.
\end{equation}
The fact that this expression is a polynomial with integer coefficients 
follows, for instance, from Corollary~\ref{cor:5.2}. If we denote it by $f(z)$,
then we easily see that $z^{2n}f(1/z)=f(z)$, which shows that $Q_n(1,z,z^2)$ is
self-reciprocal of degree $2n$. Since $f(0)=1$, this polynomial has constant
coefficient 1 and is thus also monic. The final statement follows from 
$f(1)=(3^{n+1}-1)/2$.
\end{proof}

Given the form of the numerator in \eqref{5.11}, it is not
surprising that there should be a connection between the polynomials
$Q_n(1,z,z^2)$ and trinomial coefficients. In fact, the central coefficients
1, 2, 5, 12, 31, 82, $\ldots$, are listed in \cite[A097893]{OEIS} as partial
sums of the central trinomial coefficients, with numerous properties shown
there.

The fact that the polynomials $Q_n(1,z,z^2)$ are self-reciprocal, combined with
Corollary~\ref{cor:5.4}, gives the following.

\begin{corollary}\label{cor:5.6}
Given $n\geq 1$, consider the set of all $2$-restricted binary overpartitions
of $2^{n+1}-2$, and let $j$ be such that $0\leq j\leq 2n$.  
Then the number of partitions with $j$ non-overlined parts is equal to 
those with $2n-j$ non-overlined parts.
\end{corollary}

We conclude this subsection with some results on factors and irreducibility.
As usual, $\Phi_d(x)$ will denote the $d$th cyclotomic polynomial which, by
definition, is irreducible.

\begin{proposition}\label{prop:5.7}
$(a)$ For $n\geq 1$, the polynomials $Q_n(1,z,z)$ have the following 
factorization into irreducible factors:
\begin{equation}\label{5.12}
Q_n(1,z,z)=\prod_{\substack{d\mid n+1\\d\neq 1}}
\left(z^{\varphi(d)}\Phi_d(2+z^{-1})\right).
\end{equation}
In particular, $Q_n(1,z,z)$ is irreducible if and only if $n+1$ is prime.

$(b)$ For $n\geq 1$, the polynomials $Q_n(1,z,z^2)$ have the factorization
\begin{equation}\label{5.13}
Q_n(1,z,z^2)=\prod_{\substack{d\mid n+1\\d\neq 1}}
\left(z^{\varphi(d)}\Phi_d(z+1+z^{-1})\right).
\end{equation}
As a consequence, $Q_n(1,z,z^2)$ cannot be irreducible unless $n+1$ is prime.
\end{proposition}

Proposition~\ref{prop:5.7}(a) is illustrated by Table~4.

\bigskip
\begin{center}
{\renewcommand{\arraystretch}{1.1}
\begin{tabular}{|r|l|l|}
\hline
$n$ & $Q_n(1,z,z)$ & factored \\
\hline
0 & 1 & 1\\
1 & $3z+1$ & irreducible \\
2 & $7z^2+5z+1$ & irreducible \\
3 & $15z^3 + 17z^2 + 7z + 1$ & $(3z+1)(5z^2+4z+1)$\\
4 & $31z^4 + 49z^3 + 31z^2 + 9z + 1$ & irreducible \\
5 & $63z^5 + 129z^4+111z^3+49z^2+11z+1$ & $(3z+1)(3z^2+3z+1)$\\
& & $\qquad\qquad\cdot(7z^2+5z+1)$ \\
\hline
\end{tabular}}

\medskip
{\bf Table~4}: $Q_n(1,z,z)$ for $0\leq n\leq 5$.
\end{center}

\bigskip
Related to the last statement of Proposition~\ref{prop:5.7}(b), we need to 
mention that
\begin{align}
z\Phi_2(z+1+z^{-1}) &= (z+1)^2,\label{5.14}\\
z^4\Phi_5(z+1+z^{-1}) &= \left(z^4+3z^3+4z^2+2z+1\right)
\left(z^4+2z^3+4z^2+3z+1\right). \label{5.15}
\end{align}
In particular, this means that \eqref{5.13} does not always give a complete
factorization into irreducibles, and $Q_n(1,z,z^2)$ is not always 
irreducible when $n+1$ is prime. However, we conjecture that 
$z^{\varphi(d)}\Phi_d(z+1+z^{-1})$ is irreducible for all $d\geq 1$,
$d\not\in\{2, 5\}$. We have not pursued this question further, and we refrain
from providing factorizations in Table~5. In the range of Table~5, only 
$Q_2(1,z,z^2)$ is irreducible, which is consistent with what we wrote in 
this paragraph.

\bigskip
\begin{center}
{\renewcommand{\arraystretch}{1.1}
\begin{tabular}{|r|l|}
\hline
$n$ & $Q_n(1,z,z^2)$ \\
\hline
0 & 1 \\
1 & $z^2+2z+1$ \\
2 & $z^4+3z^3+5z^2+3z+1$ \\
3 & $z^6 + 4z^5 + 9z^4 + 12z^3 + 9z^2 + 4z + 1$ \\
4 & $z^8 + 5z^7 + 14z^6 + 25z^5 + 31z^4 + 25z^3 + 14z^2 + 5z + 1$ \\
5 & $z^{10} + 6z^9 + 20z^8 + 44z^7 + 70z^6 + 82z^5 + 70z^4+44z^3+20z^2+6z+1$ \\
\hline
\end{tabular}}

\medskip
{\bf Table~5}: $Q_n(1,z,z^2)$ for $0\leq n\leq 5$.
\end{center}

\bigskip
\begin{proof}[Proof of Proposition~\ref{prop:5.7}]
We use the well-known identity
\begin{equation}\label{5.16}
w^{n+1}-1 = \prod_{d\mid n+1}\Phi_d(w),
\end{equation}
substitute $w=(2z+1)/z$, and multiply both sides by $z^{n+1}$. This gives
\begin{equation}\label{5.17}
(2z+1)^{n+1}-z^{n+1} 
= \prod_{d\mid n+1}\left(z^{\varphi(d)}\Phi_d(2+z^{-1})\right),
\end{equation}
where we have used the identity $\sum_{d\mid n+1}\varphi(d)=n+1$. Since 
$\Phi_1(w)=w-1$, we have $z\Phi_1(2+z^{-1})=z+1$, so \eqref{5.12} follows from
dividing both sides of \eqref{5.17} by $z+1$ and using the identity
\begin{equation}\label{5.18}
Q_n(1,z,z)=\frac{z^{n+1}}{z+1}\left(\left(\frac{2z+1}{z}\right)^{n+1}-1\right),
\end{equation}
which follows easily from \eqref{5.2}. We also
note that $\Phi_d(w)$ has degree $\varphi(d)$, so all factors
$z^{\varphi(d)}\Phi_d(2+z^{-1})$ are indeed polynomials.

Next, since $\Phi_d(w)$ is irreducible, the linear shift $\Phi_d(2+z^{-1})$
gives an irreducible polynomial in $z^{-1}$, and 
$z^{\varphi(d)}\Phi_d(2+z^{-1})$ is then irreducible as a polynomial in $z$.
This shows that \eqref{5.12} is a complete factorization into 
irreducible factors and that $Q_n(1,z,z)$ is irreducible when $n+1$ is prime.

The identity \eqref{5.13} can be obtained in exactly the same way as 
\eqref{5.12}, with the only difference that in this case we have
$z\Phi_1(z+1+z^{-1})=z^2+1$. The final statement in part (b) follows from the
fact that the right-hand side of \eqref{5.13} consists of just one term.
\end{proof}

\section{Zero distributions of $Q_n(1,z,z)$ and $Q_n(1,z,z^2)$}\label{sec:6}

Since $Q_n(1,z,z)$ and $Q_n(1,z,z^2)$ are single-variable polynomials, 
it makes sense to consider their zero distribution. We begin with the
easier case.

\begin{proposition}\label{prop:6.1}
For any $n\geq 1$, the zeros of $Q_n(1,z,z)$ are given by
\begin{equation}\label{6.1}
z_j = \frac{1}{\zeta_j-2},\qquad j=1,2,\ldots,n,
\end{equation}
where $\zeta_j:=e^{2\pi ij/(n+1)}$ is an $(n+1)$th root of unity.
Furthermore, the zeros $z_j$ all lie on the circle given by
\begin{equation}\label{6.2}
\left(x+\tfrac{2}{3}\right)^2 + y^2 = \left(\tfrac{1}{3}\right)^2.
\end{equation}
\end{proposition}

\begin{proof}
We consider the identity \eqref{5.18} and note that
the term in large parentheses is zero exactly when $(2z+1)/z=\zeta_j$,
$j=0,1,\ldots,n$, which is equivalent to $z=1/(\zeta_j-2)$. When $j=0$, we get
$z=-1$; however, a limit argument shows that $z=-1$ is not a zero of the
right-hand side of \eqref{5.18}. This proves \eqref{6.1}.

For the second statement, we note that by the theory of fractional linear
transformations all the $z_j$ lie on a circle (or a straight line) since all
$\zeta_j$ lie on a circle. It therefore suffices to show that any three
distinct points $\zeta\in{\mathbb C}$ with $|\zeta|=1$ are mapped to points
$z=x+iy$ satisfying \eqref{6.2}. It is easy to verify this with the choice of
$\zeta=-1,\pm i$, for instance; this completes the proof.  
\end{proof}

The zero distribution of $Q_n(1,z,z^2)$ turns out to be more interesting than
that of $Q_n(1,z,z)$. We begin with a lemma.

\begin{lemma}\label{lem:6.2}
For any $n\geq 1$, the $2n$ zeros of $Q_n(1,z,z^2)$ are given by
\begin{equation}\label{6.3}
z_j^{\pm}:=\frac{1}{2}
\left(\zeta_j-1+\left(\zeta_j^2-2\zeta_j-3\right)^{\frac{1}{2}}\right),\quad
j=1, 2,\ldots, n,
\end{equation}
where $\zeta_j:=e^{2\pi ij/(n+1)}$ are the $(n+1)$th roots of unity and the
superscript $\pm$ indicates that for each $j$ we have two values of \eqref{6.3}.
\end{lemma}

\begin{proof}
Upon slightly rewriting \eqref{5.11}, we have
\begin{equation}\label{6.4}
Q_n(1,z,z^2) 
= \frac{z^{n+1}}{z^2+1}\left(\left(z+1+\frac{1}{z}\right)^{n+1}-1\right).
\end{equation}
For this expression to vanish, we need 
\begin{equation}\label{6.5}
z+1+\frac{1}{z} = \zeta_j,\quad\hbox{or}\quad z^2+(1-\zeta_j)z+1=0,\qquad
j=0, 1, \ldots, n.
\end{equation}
Solving this last equation, we get \eqref{6.3} with $j=0, 1, \ldots, n$.
However, since $\zeta_0=1$, we have $z_0^{\pm}=\pm i$, a pair of solutions 
which is canceled by the denominator in \eqref{6.4}. Hence $j=0$ needs to be
excluded, which completes the proof.
\end{proof}

In the paragraph after \eqref{6.5} we saw that $\pm i$ cannot be zeros of
$Q_n(1,z,z^2)$ for any $n\geq 1$. On the other hand, by substituting $z=-1$
in \eqref{6.4} or \eqref{5.11} and in the derivative of the numerator of 
\eqref{5.11}, we see that $z=-1$ is a double zero of $Q_n(1,z,z^2)$ if and
only if $n$ is odd. This corresponds to $\zeta_{(n+1)/2}=-1$, along with 
\eqref{6.3} and/or \eqref{6.5}.

The zeros $z_j^{\pm}$ for $n=21$ and $n=50$ are shown in Figure~1. They lie on
an algebraic curve which we will identify next.

\begin{proposition}\label{prop:6.3}
The zeros of all polynomials $Q_n(1,z,z^2)$, $n\geq 1$, lie on the algebraic
curve
\begin{equation}\label{6.6}
x^4+2x^2y^2+y^4+2x^3+2xy^2+2x^2-2y^2+2x+1=0
\end{equation}
or, rewritten,
\begin{equation}\label{6.7}
\big(x^2+y^2+x\big)^2 + (x+1)^2 = 2y^2.
\end{equation}
\end{proposition}

\begin{proof}
If $z\in{\mathbb C}$ is a zero of $Q_n(1,z,z^2)$ for any integer $n\geq 1$,
then by the left identity in \eqref{6.5} we have
\begin{equation}\label{6.8}
\left|z+1+z^{-1}\right|^2 = 1.
\end{equation}
Setting $z=x+iy$ ($x,y\in{\mathbb R}$), so that $z^{-1}=(x-iy)/(x^2+y^2)$,
we get with \eqref{6.8}, 
\begin{align*}
1&=\left|x+iy+1+\frac{x-iy}{x^2+y^2}\right|^2\\
&=\left(x+1+\frac{x}{x^2+y^2}\right)^2+\left(y-\frac{y}{x^2+y^2}\right)^2\\
&=\frac{1}{x^2+y^2}\left(x^4+2x^2y^2+y^4+2x^3+2xy^2+3x^2-y^2+2x+1\right).
\end{align*}
This last line implies
\[
x^4+2x^2y^2+y^4+2x^3+2xy^2+3x^2-y^2+2x+1=x^2+y^2,
\]
which is equivalent to \eqref{6.6} and \eqref{6.7}.
\end{proof}

\bigskip
\begin{center}
$\includegraphics[scale=0.35]{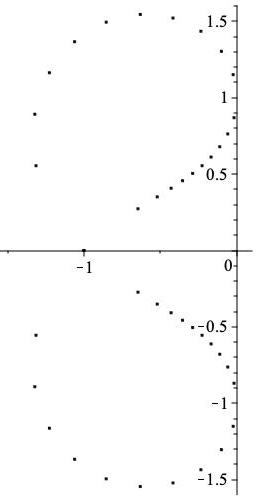}$\qquad\quad
$\includegraphics[scale=0.35]{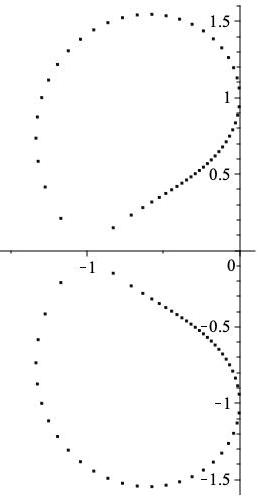}$\qquad\quad
$\includegraphics[scale=0.35]{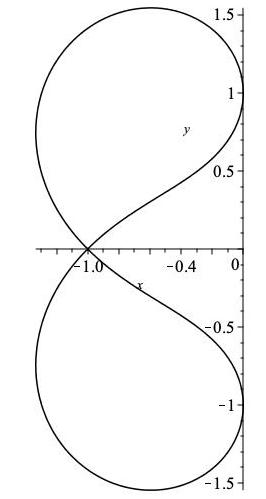}$

\medskip
{\bf Figure~1}: The zeros of $Q_{21}(1,z,z^2)$, $Q_{50}(1,z,z^2)$, and the
curve \eqref{6.6}.
\end{center}

\medskip
We recall that in most of Sections~\ref{sec:5} and~\ref{sec:6} we have
considered the polynomial sequence $Q_n(1,y,z)$ from \eqref{5.2} with $y$
replaced by $z^\alpha$ and $z$ replaced by $z^\beta$, and then studied the 
two specific cases $(\alpha,\beta)=(1,1)$ and $(\alpha,\beta)=(1,2)$. This 
suggests that we might as well consider the general case with integers $\alpha$,
$\beta$, not both zero. From \eqref{5.2} we then obtain
\begin{equation}\label{6.9}
Q_n(1,z^\alpha,z^\beta) = \frac{z^{(n+1)\alpha}}{z^\beta+1}
\left(\left(1+z^{\beta-\alpha}+z^{-\alpha}\right)^{n+1}-1\right).
\end{equation}

Just as we did in the proofs of Proposition~\ref{prop:6.1} 
and/or~\ref{prop:6.3}, for any pair of integers $\alpha$, $\beta$ (not both 
zero) we can determine the algebraic curve (considered as a curve in the 
complex plane) on which all
the zeros of $Q_n(1,z^\alpha,z^\beta)$ lie. We denote these curves by
$f_{\alpha,\beta}(x,y)=0$ and display the polynomials $f_{\alpha,\beta}(x,y)$ 
in Table~6, for the first few nonnegative $\alpha,\beta$. This table also shows
the corresponding genus in each case, which was computed using Maple.
It would also be possible to derive a general formula for 
$f_{\alpha,\beta}(x,y)$; however, this will not be required here.

\bigskip
\begin{center}
{\renewcommand{\arraystretch}{1.1}
\begin{tabular}{|r|r|l|r|}
\hline
$\alpha$ & $\beta$ & $f_{\alpha,\beta}(x,y)$ & $g$ \\
\hline
0 & 1 & $3+4x+x^2+y^2$ & 0 \\
0 & 2 & $3+4x^2+x^4-4y^2+2x^2y^2+y^4$ & 1 \\
0 & 3 & $3+4x^3+x^6-12xy^2+3x^4y^2+3x^2y^4+y^6$ & 4 \\
\hline
1 & 0 & $1+x$ & 0 \\
1 & 1 & $1+4x+3x^2+3y^2$ & 0 \\
1 & 2 & $1+2x+2x^2+2x^3+x^4-2y^2+2xy^2+2x^2y^2+y^4$ & 0 \\
1 & 3 & $1+2x+2x^3+2x^4+x^6-6xy^2+3x^4y^2-2y^4+3x^2y^4+y^6$ & 4 \\
\hline
2 & 0 & $1+x^2-y^2$ & 0 \\
2 & 1 & $1+2x+3x^2+2x^3-y^2+2xy^2$ & 1 \\
2 & 2 & $1+4x^2+3x^4-4y^2+6x^2y^2+3y^4$ & 1 \\
2 & 3 & $1+2x^2+2x^3+2x^5+x^6-2y^2-6xy^2+4x^3y^2+3x^4y^2+2xy^4$ & \\
  &   & $\qquad +3x^2y^4+y^6$ & 4 \\
\hline
3 & 0 & $1+x^3-3xy^2$ & 1 \\
3 & 1 & $1+2x+x^2+2x^3+2x^4+y^2-6xy^2-2y^4$ & 3 \\
3 & 2 & $1+2x^2+2x^3+x^4+2x^5-2y^2-6xy^2+2x^2y^2+4x^3y^2+y^4+2xy^4$ & 4 \\
3 & 3 & $1+4x^3+3x^6-12xy^2+9x^4y^2+9x^2y^4+3y^6$ & 4 \\
\hline
\end{tabular}}

\medskip
{\bf Table~6}: $f_{\alpha,\beta}(x,y)$ for $0\leq \alpha,\beta\leq 3$ and
genus $g$ of curve $f_{\alpha,\beta}(x,y)=0$.
\end{center}

\bigskip
We now state an easy transformation identity, which follows directly from 
\eqref{6.9}.

\begin{lemma}\label{lem:6.4}
For all integers $\alpha,\beta$, not both zero, we have
\begin{equation}\label{6.10}
Q_n(1,z^{\beta-\alpha},z^\beta) = 
z^{n\beta}Q_n(1,(\tfrac{1}{z})^\alpha,(\tfrac{1}{z})^\beta).
\end{equation}
\end{lemma}

The identity \eqref{6.10}, and thus the relationship  between the curves 
$f_{\alpha,\beta}(x,y)=0$ and $f_{\beta-\alpha,\beta}(x,y)=0$,
means that if $z=x+iy$ lies on one of these curves, then
$1/z={\overline z}/(x^2+y^2)$ and by symmetry also $1/{\overline z}=z/|z|^2$ 
lie on its companion. This implies that we have an inversion with respect to
the unit circle: if a
point on one curve has polar coordinates $r(\cos\theta+i\,\sin\theta)$, then
the corresponding point on its companion has polar coordinates 
$r^{-1}(\cos\theta+i\,\sin\theta)$. We say, in short, that the two curves in
\eqref{6.10} are {\it inverse} to each other.

\begin{example}\label{ex:7.1}
{\rm (a) The curve $f_{1,2}(x,y)=0$, given explicitly in 
Proposition~\ref{prop:6.3} and shown in Figure~1, is its own inverse.

(b) As we saw in Proposition~\ref{prop:6.1}, the curve $f_{1,1}(x,y)=0$ 
is the circle of radius $1/3$ centered at $(x,y)=(-2/3,0)$. Its inverse
is $f_{0,1}(x,y)=3+4x+x^2+y^2=0$, i.e., the circle $(x+2)^2+y^2=1$.
The two circles are tangent to each other at $(x,y)=(-1,0)$.

(c) Using again similar methods as in the proofs of Propositions~\ref{prop:6.1}
and~\ref{prop:6.3}, we find $f_{1,0}(x,y)=1+x$ and 
$f_{-1,0}(x,y)=(2x+1)^2+4y^2-1$. The corresponding mutual inverses
are then the vertical line $x=-1$ and the circle 
$(x+\tfrac{1}{2})^2+y^2=(\tfrac{1}{2})^2$, which are also tangent to each 
other at $(x,y)=(-1,0)$.

(d) Using once again the same methods as before, we find
\begin{align}
f_{2,1}(x,y) &= 1+2x+3x^2+2x^3-y^2+2xy^2,\label{6.11}\\
f_{-1,1}(x,y) &= 2x+3x^2+2x^3+x^4-y^2+2xy^2+2x^2y^2+y^4;\label{6.12}
\end{align} 
see Figure~2 for a joint plot.}
\end{example}

\bigskip
\begin{center}
$\includegraphics[scale=0.35]{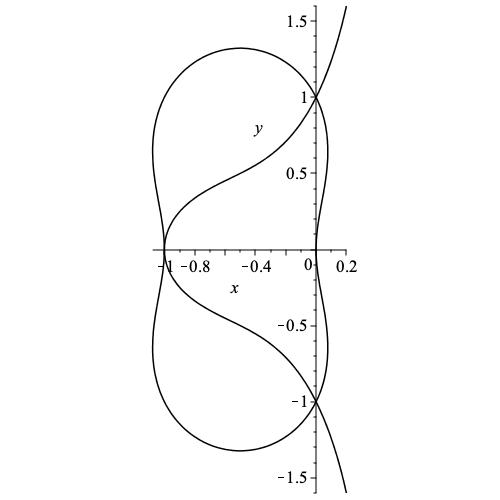}$

\medskip
{\bf Figure~2}: The curves $f_{2,1}(x,y)=0$ and $f_{-1,1}(x,y)=0$.
\end{center}

\medskip
We have not been able to identify the cubic $f_{2,1}(x,y)=0$ as a specific 
known curve. However, one can easily see that it has a vertical asymptote at
$x=1/2$. Its inverse $f_{-1,1}(x,y)=0$ is an oval of Cassini,
as can be seen by rewriting this last equation in the form
\begin{equation}\label{6.13}
\left(\left(x+\tfrac{1}{2}\right)^2+\left((y-\tfrac{1}{2}\sqrt{3}\right)^2\right)
\left(\left(x+\tfrac{1}{2}\right)^2+\left((y+\tfrac{1}{2}\sqrt{3}\right)^2\right)=1.
\end{equation}
The equation \eqref{6.13} shows that the product of the distances between a 
point $(x,y)$ on the curve and the two foci 
$(-\frac{1}{2},\pm\frac{1}{2}\sqrt{3})$ is always 1, which is consistent with
the definition of an oval of Cassini; see, e.g., \cite{Wi} or \cite{La}.

\section{Further properties of the curve \eqref{6.6}}\label{sec:7}

A particularly interesting example of the algebraic curves containing the 
zeros of $Q_n(1,z^\alpha,z^\beta)$ is the one belonging to
$(\alpha,\beta)=(1,2)$. This curve was obtained in Proposition~\ref{prop:6.3};
see also Figure~1. Using Maple, we found that it has genus 0; see also Table~6.
This means that the curve has a rational parametrization, which can also be 
found by Maple, using the package {\tt algcurves} with the functions 
{\tt genus} and {\tt parametrization}. 

\begin{proposition}\label{prop:7.1} The curve $f_{1,2}(x,y)=0$ has the 
following parametrization with rational functions:
\[
x =-\frac{4t^2(4t+1)^2}{169t^4+178t^3+74t^2+14t+1}, \qquad
y = \frac{(5t+1)(3t+1)(17t^2+8t+1)}{169t^4+178t^3+74t^2+14t+1}.
\]
\end{proposition}

By choosing the parameters $t=-1/5$ and $t=-1/3$ in Proposition~\ref{prop:7.1},
we obtain the double point $(x,y)=(-1,0)$ on the curve. Furthermore, by 
evaluating the quotients of the derivatives $dy/dt$ and $dx/dt$ at $t=-1/5$ 
and $t=-1/3$, we see that the slope of the curve at this double point is 1,
resp.\ $-1$. Apart from the special parameters $t=0$ and $t=-1/4$, it is worth
mentioning that as $t\rightarrow\pm\infty$, we have $x\rightarrow -(8/13)^2$
and $y\rightarrow 15\cdot 17/13^2$.

Next, we derive some maximum/minimum properties of the curve in question.

\begin{proposition}\label{prop:7.2} 
$(a)$ The curve $f_{1,2}(x,y)=0$ has vertical tangents exactly at the points
$(x,y)=(0,\pm 1)$ and $(-\frac{4}{3},\pm\frac{1}{3}\sqrt{5})$.

$(b)$ The curve has horizontal tangents exactly at the points 
$(x,y)=(x_0,\pm y_0)$,
where
\begin{align}
x_0 &=-\frac{a}{18}+\frac{22}{9a}-\frac{4}{9}\quad\hbox{with}\quad
a=\sqrt[3]{188+36\sqrt{93}},\label{7.1} \\
y_0 &=\frac{1}{198}\sqrt{(495\sqrt{93}-1617)a+33396+(-18\sqrt{93}+699)a^2}.\label{7.2} 
\end{align}
Numerically, $(x_0,y_0)\simeq (-0.594414, 1.545634)$.
\end{proposition}

\begin{proof}
Implicit differentiation of \eqref{6.6} leads to 
\begin{equation}\label{7.3}
\frac{dy}{dx}=\frac{1+2x+3x^2+2x^3+(1+2x)y^2}{2y(1-x-x^2-y^2)},
\end{equation}
while solving \eqref{6.6} for $y^2$ gives
\begin{equation}\label{7.4}
y^2=1-x-x^2\pm\sqrt{-x(4+3x)}.
\end{equation}
Considering the denominator of \eqref{7.3}, we first note that $y=0$ implies
$x=-1$, and thus the numerator will also vanish (see also the remark following
Proposition~\ref{prop:7.1}). When $1-x-x^2-y^2=0$, then by \eqref{7.4} we have
$x=0$ or $x=-4/3$. Substituting this back into \eqref{7.4}, we find $y=\pm 1$,
resp.\ $y=\pm\frac{1}{3}\sqrt{5}$, which completes the proof of part (a).

Next, in order to find the points on the curve which have horizontal tangents,
we substitute \eqref{7.4} into the numerator of \eqref{7.3}. After some
straightforward manipulations we find that this numerator vanishes if and 
only if $1+4x+7x^2+7x^3+3x^4=0$, which factors as 
\[
(x+1)(3x^3+4x^2+3x+1)=0.
\]
The solution $x=-1$ gives $y=0$, a case we already discussed, while the unique
real root of the cubic is $x_0$ as given in \eqref{7.1}. This solution was 
obtained with the help of Maple. 

Finally, substituting $x=x_0$ into \eqref{6.6}, we get a quartic polynomial
in $y$. With the help of Maple we can solve this algebraically and note that
the largest root is $y_0$ as given in \eqref{7.2}. This completes the proof
of part (b).
\end{proof}

Finally in this section, we will see that the maximal and minimal moduli of 
the curve in \eqref{6.6} have remarkably simple forms.

\begin{proposition}\label{prop:7.3}
The points on the curve $f_{1,2}(x,y)=0$ have maximal modulus $\sqrt{3}$ and 
minimal modulus $1/\sqrt{3}$. Both are attained when $\cos\theta=-1/\sqrt{3}$,
or numerically when $\theta\simeq\pm 0.695913\,\pi$.
The Cartesian coordinates of these extremal points are 
$(x,y)=(-1,\pm\sqrt{2})$, resp.\ $(x,y)=(-1/3,\pm\sqrt{2}/3)$.
\end{proposition}

\begin{proof}
As in the proof of Proposition~\ref{prop:6.3}, we begin with the identity
\eqref{6.8}, but this time we use the polar representations
\[
z=r(\cos\theta+i\,\sin\theta),\qquad z^{-1}=r^{-1}(\cos\theta-i\,\sin\theta).
\]
After some straightforward manipulations we see that \eqref{6.8} is equivalent
to 
\[
\left(r+\tfrac{1}{r}\right)^2\cos^2\theta+2\left(r+\tfrac{1}{r}\right)
\cos\theta+\left(r-\tfrac{1}{r}\right)^2\sin^2\theta = 0,
\]
which can be further transformed to 
\begin{equation}\label{7.5}
\left(r+\tfrac{1}{r}\right)^2+2\left(r+\tfrac{1}{r}\right)\cos\theta
+4\cos^2\theta-4=0.
\end{equation}
We solve \eqref{7.5} for $r+r^{-1}$, obtaining
\begin{equation}\label{7.6}
r+\frac{1}{r} = -\cos\theta+\sqrt{4-3\cos^2\theta},\qquad
\frac{\pi}{2}\leq |\theta|\leq\pi,
\end{equation}
where the restriction on $\theta$ comes from the fact that the curve in
question lies entirely in the left half-plane. Furthermore, we have just
``$+$" in front of the square root in \eqref{7.6} since
$\sqrt{4-3\cos^2\theta}\geq 1$ for all $\theta$, while the left-hand side of
\eqref{7.6} is $\geq 2$.

Differentiating both sides of \eqref{7.6}, we get after some simplification,
\begin{equation}\label{7.7}
2r\left(1-r^{-2}\right)\frac{dr}{d\theta} 
= \left(1+\frac{3\cos\theta}{\sqrt{4-3\cos^2\theta}}\right)\sin\theta. 
\end{equation}
The right-hand side of \eqref{7.7} vanishes when $\sin\theta=0$ or the 
expression in large parentheses is zero. In the first case we get $\theta=\pi$,
which means $r=1$ and thus the left-hand side of \eqref{7.7} also vanishes.
But this has already been dealt with following Proposition~\ref{prop:7.1}.

In the second case we have
\begin{equation}\label{7.8}
-3\cos\theta = \sqrt{4-3\cos^2\theta},
\end{equation}
and upon squaring and simplifying we get $\cos^2\theta=1/3$. But then, by
\eqref{7.8}, only the solution $\cos\theta=-1/\sqrt{3}$ is possible.
Substituting this into \eqref{7.6}, we get
\[
r+\frac{1}{r} = \frac{4}{3}\sqrt{3},
\]
which has the two solutions $r=\sqrt{3}$ and $r=1/\sqrt{3}$. Using \eqref{7.7},
for instance, we can see that these two values are a maximum and a minimum,
respectively. 

The final statement comes from the fact that $\cos\theta=-1/\sqrt{3}$ implies
$\sin\theta=\pm\sqrt{2/3}$, and thus $(x,y)=(-r/\sqrt{3},\pm r\sqrt{2/3})$.
\end{proof}

In concluding this section, we note that the curve we investigated here has
several properties in common with the {\it Besace curve} given by the equation
\[
\left(x^2-by\right)^2 = a^2\left(x^2-y^2\right),
\]
with positive parameters $a$ and $b$; see, e.g., \cite{Fe}. These similarities
include the facts that both are quartics, have genus 0, and are similar in
shape.

\section{General bases $b\geq 2$}\label{sec:8}

Much of what we did in Sections~\ref{sec:2} to~\ref{sec:6} has direct analogues
for integer bases $b\geq 2$. We therefore structure this section roughly along 
the lines of of previous sections. Most proofs are similar to those of the 
case $b=2$; we leave the details to the interested reader.

\subsection{Basic properties}

In analogy to the beginning of Section~\ref{sec:2}
we specialize the more general multicolor $b$-ary partitions in \cite{DE10} to
the 2-color $(1,b)$-case. That is, we consider $b$-ary overpartitions where the 
non-overlined parts occur at most $b$ times. As we did in \eqref{2.1} above, 
we use the following simplified notation: for all $n\geq 0$ we set
\begin{equation}\label{8.1}
p_n(Z) := \Omega_{b,T}^{(1,b)}(n;Z),\qquad
Z=(x,y_1,\ldots,y_b),
\end{equation}
with $T=(1,1,\ldots,1)$; see again \cite{DE10}. Then, in analogy to
\eqref{2.2}, we have the generating function
\begin{equation}\label{8.2}
\sum_{n=0}^\infty p_n(Z)q^n
=\prod_{j=0}^\infty\left(1+xq^{b^j}\right)
\left(1+y_1q^{b^j}+y_2q^{2\cdot b^j}+\cdots+y_bq^{b\cdot b^j}\right),
\end{equation}
and in analogy to \eqref{2.3}, \eqref{2.4} we have the recurrence relations
with initial terms
\begin{equation}\label{8.3}
p_0(Z)=1,\quad p_1(Z)=x+y_1,\quad p_j(Z)=xy_{j-1}+y_j\quad (2\leq j\leq b-1),
\end{equation}
and for $n\geq 1$,
\begin{align}
p_{bn}(Z)&=p_n(Z)+(y_b+xy_{b-1})\cdot p_{n-1}(Z),\label{8.4}\\
p_{bn+1}(Z)&=(x+y_1)\cdot p_n(Z)+xy_b\cdot p_{n-1}(Z),\label{8.5}\\
p_{bn+j}(Z)&=(xy_{j-1}+y_1)\cdot p_n(Z),\quad (2\leq j\leq b-1).\label{8.6}
\end{align}
Since for $b=2$ we have $Z=(x,y_1,y_2)=(x,y,z)$, it is clear that \eqref{8.4}
and \eqref{8.5} become \eqref{2.3} and \eqref{2.4}, respectively, while 
\eqref{8.6} occurs only for $b\geq 3$.

Next we state the base-$b$ analogue of Proposition~\ref{prop:2.1}. If we write
the polynomials $p_n(Z)$ in the form
\begin{equation}\label{8.7}
p_n(Z) = \sum_{i,j_1,\ldots,j_b\geq 0}
c_n(i,j_1,\ldots,j_b)\cdot x^iy_1^{j_1}\ldots y_b^{j_b},\quad n\geq 0,
\end{equation}
then the generating function \eqref{8.2} gives
the following combinatorial interpretation.

\begin{proposition}\label{prop:8.1}
For any non-negative integers $n,i,j_1,\ldots,j_b$, the coefficient \\
$c_n(i,j_1,\ldots,j_b)$ in \eqref{8.7} counts the number of $b$-restricted 
$b$-ary overpartitions of $n$ that have
\begin{enumerate}
\item[] $i$ different and single overlined parts,
\item[] $j_1$ different and single non-overlined parts, and
\item[] $j_k$ different k-tuples of non-overlined parts, $2\leq k\leq b$.
\end{enumerate}
\end{proposition}

\subsection{Connections with Chebyshev polynomials}

We now extend the results in Section~\ref{sec:3} to arbitrary bases $b\geq 2$.
We consider two subsequences with subscripts
\begin{equation}\label{8.8}
q(n):=\frac{b^{n+1}-b}{b-1},\qquad r(n):=\frac{b^n-1}{b-1} = \frac{q(n)}{b},
\end{equation} 
which we use to define
\begin{equation}\label{8.9}
Q_n^b(Z):=p_{q(n)}(Z),\qquad R_n^b(Z):=p_{r(n)}(Z).
\end{equation} 
For $b=2$ and $Z=(x,y,z)$, the identities \eqref{8.8}, \eqref{8.9} and 
\eqref{3.1} then give 
\[
Q_n^2(Z)=Q_n(x,y,z),\qquad R_n^2(Z)=R_n(x,y,z).
\]
The following is analogous to Proposition~\ref{prop:3.1}. To simplify notation,
we set for $b\geq 2$ and $x,y_1,y_{b-1},y_b$ as in \eqref{8.1},
\begin{equation}\label{8.10}
W_1^b(Z) := xy_{b-1}+x+y_1+y_b,\qquad
W_2^b(Z) := x^2y_{b-1}+xy_1y_{b-1}+y_1y_b.
\end{equation}

\begin{proposition}\label{prop:8.2}
We have $Q_0^b(Z)=1$, $Q_1^b(Z)=W_1^b(Z)$, $R_0^b(Z)=1$, $R_1^b(Z)=x+y_1$, 
and for $n\geq 1$,
\begin{align}
Q_{n+1}^b(Z)&=W_1^b(Z)\cdot Q_n^b(Z)-W_2^b(Z)\cdot Q_{n-1}^b(Z),\label{8.11}\\
R_{n+1}^b(Z)&=W_1^b(Z)\cdot R_n^b(Z)-W_2^b(Z)\cdot R_{n-1}^b(Z).\label{8.12}
\end{align}
\end{proposition}
Further in analogy to Section~\ref{sec:3}, the recurrence relations 
\eqref{8.11}, \eqref{8.12} lead to the following.

\begin{proposition}\label{prop:8.3}
The polynomials $Q_n^b$ and $R_n^b$ satisfy the generating functions
\begin{align}
\sum_{n=0}^{\infty}Q_n^b(Z)q^n
=\frac{1}{1-W_1^b(Z)q+W_2^b(Z)q^2},\label{8.13}\\
\sum_{n=0}^{\infty}R_n^b(Z)q^n
=\frac{1-(xy_{b-1}+y_b)q}{1-W_1^b(Z)q+W_2^b(Z)q^2}.\label{8.14}
\end{align}
\end{proposition}

At this point it will not be surprising that the polynomials $Q_n^b$ and
$R_n^b$ are also closely related to the Chebyshev polynomials of both kinds;
see \eqref{3.10}--\eqref{3.12}.

\begin{proposition}\label{prop:8.4}
For all $n\geq 0$ we have
\begin{align}
Q_n^b(Z)&=\left(W_2^b(Z)\right)^{n/2}
U_n\left(\frac{W_1^b(Z)}{2\sqrt{W_2^b(Z)}}\right),\label{8.15}\\
R_n^b(Z)&=\left(W_2^b(Z)\right)^{n/2}
T_n\left(\frac{W_1^b(Z)}{2\sqrt{W_2^b(Z)}}\right)
+\widetilde{U}_{n-1}^b(Z),\label{8.16}
\end{align}
where
\begin{equation}\label{8.17}
\widetilde{U}_{n-1}^b(Z)
= \frac{x+y_1-xy_{b-1}-y_b}{2}\cdot Q_{n-1}^b(Z).
\end{equation}
\end{proposition}

By combining the identities \eqref{8.15}--\eqref{8.17} with the well-known
relation $T_n(x)=U_n(x)-xU_{n-1}(x)$, we obtain the following identity.

\begin{corollary}\label{cor:8.5}
For $n\geq 1$ we have
\begin{equation}\label{8.18}
R_n^b(Z)=Q_n^b(Z)-\left(xy_{b-1}+y_b\right)\cdot Q_{n-1}^b(Z).
\end{equation}
\end{corollary}

Likewise, the identity \eqref{3.1a} extends to $b \geq 2$ with 
$y\rightarrow y_1$ and $z \rightarrow y_b$. 
The following fact that is a consequence of any one of 
Propositions~\ref{prop:8.2}--\ref{prop:8.4}.

\begin{corollary}\label{cor:8.6}
For all $b\geq 2$ and $n\geq 0$, $Q_n^b(Z)$ and $R_n^b(Z)$ are polynomials in 
$x$, $y_1$, $y_{b-1}$, and $y_b$ only. 
\end{corollary}

\begin{example}\label{ex:9.1}
{\rm When $b=5$, then by \eqref{8.8} we have $q(1)=5$ and 
$q(2)=30$. Accordingly, we get $Q_1^5(Z)=p_5(Z)=y_1+x+y_5+xy_4$; see also
Proposition~\ref{prop:8.2}. This polynomial corresponds to the 5-restricted
$b$-ary overpartitions (with $b=5$)
\[
(5),\;(\overline{5}),\;(1,1,1,1,1),\;(\overline{1},1,1,1,1),
\]
written in the order of the coefficients of $Q_1^5(Z)$. This is consistent 
with Proposition~\ref{prop:8.1}. }
\end{example}

\begin{example}\label{ex:9.2}
{\rm Similarly we find, for instance with \eqref{8.11}, that $Q_2^5(Z)$ is 
\[
p_{30}(Z)=y_1^2+2xy_1+x^2+y_1y_5+2xy_5+xy_1y_4+x^2y_4+y_5^2+2xy_4y_5+x^2y_4^2.
\] 
With the usual notation of $a^k$ for the part $a$ repeated $k$ times,
the corresponding 5-restricted $b$-ary overpartitions (with $b=5$) are
\begin{align*}
&(25,5),\;(\overline{25},5),\;(25,\overline{5}),\;(\overline{25},\overline{5}),
\;(25,1^5),\;(\overline{25},1^5), (\overline{5},5^5), \\
&(25,\overline{1},1^4),\; (\overline{25},\overline{1},1^4),\;(5^5,1^5),\;
(\overline{5},5^4,1^5),\;(5^5,\overline{1},1^4),\;
(\overline{5},5^4,\overline{1},1^4),
\end{align*}
again in the order of the terms of $Q_2^5(Z)$. For instance, the two partitions
corresponding to the monomial $2xy_1$ are $(\overline{25},5)$ and
$(25,\overline{5})$, and the two corresponding to $2xy_5$ are
$(\overline{25},1^5)$ and $(\overline{5},5^5)$. All this is again
consistent with Proposition~\ref{prop:8.1}.}
\end{example}

\subsection{A first special case}

It is no coincidence that the number of $b$-restricted $b$-ary overpartitions
in Example~\ref{ex:2.1} (for $b=2$ and $n=6$) is the same as that in 
Example~\ref{ex:9.1} 
(for $b=5$ and $n=30$), namely 13. Indeed, if we set $x=y_1=\cdots =y_b=1$, 
then any one of Propositions~\ref{prop:8.2}--\ref{prop:8.4}, together with 
Corollary~\ref{cor:4.2}(d), imply the following.

\begin{corollary}\label{cor:8.7}
Let $b\geq 2$ be an integer, and $q(n), r(n)$ the sequences defined by
\eqref{8.8}. Then for each $n\geq 0$ the number of $b$-restricted $b$-ary 
overpartitions of $q(n)$ and $r(n)$ are $\frac{1}{2}(3^{n+1}-1)$ and
$\frac{1}{2}(3^n+1)$, respectively. 
\end{corollary}

\begin{example}\label{ex:9.3}
{\rm We take $b=5$ again, but in contrast to Example~\ref{ex:9.1} we consider
the $b$-restricted $b$-ary overpartitions of $r(2)=6$. There are 
$\frac{1}{2}(3^2+1)=5$ of them, namely
\[
(5,1),\;(\overline{5},1),\;(5,\overline{1}),\;(\overline{5},\overline{1}),\; 
(\overline{1},1^5).
\]
The corresponding polynomial, best obtained with \eqref{8.12}, is
\[
R_2^5(Z) = y_1^2+2xy_1+x^2+xy_5,
\]
with the monomials again in the same order as the corresponding overpartitions.}
\end{example}

If we set $y_1=y_{b-1}=y_b=1$, then by \eqref{8.10} and \eqref{8.17} we have 
\[
W_1^b(Z)=2x+2,\qquad W_2^b(Z)=x^2+x+1,\qquad \widetilde{U}_{n-1}^b(Z)=0,
\]
and consequently Proposition~\ref{prop:8.4} and Corollary~\ref{cor:4.1} give
\[
Q_n^b(Z)=Q_n(x)\qquad\hbox{and}\qquad R_n^b(Z)=R_n(x).
\]
When $b\geq 4$, this is independent of $y_2,\ldots,y_{b-2}$. The case 
$y_1=y_{b-1}=y_b=1$ is therefore covered by Section~\ref{sec:4}.

\subsection{A second special case}

In analogy to Section~\ref{sec:9} we set $x=y_1=y_{b-1}=y_b$. If we rename this
common variable as $z$, then by \eqref{8.10} we have 
\[
W_1^b(Z) := z^2+3z,\qquad W_2^b(Z) := 2z^3+z^2.
\]
By Proposition~\ref{prop:8.2} we then have $Q_0^b(Z)=1$, $Q_1^b(Z)=z^2+3z$, 
and for $n\geq 1$,
\[
Q_{n+1}^b(Z)=(z^2+3z)\cdot Q_n^b(Z)-(2z^3+z^2)\cdot Q_{n-1}^b(Z),
\]
with an analogous recurrence relation also for the polynomials $R_n^b(Z)$.
But this is exactly the situation of Section~\ref{sec:9}, beginning with
Corollary~\ref{cor:9.1}. In particular, it means that we have independence of
the base $b\geq 2$.

For a general combinatorial interpretation of the polynomials 
$Q_n(Z)=Q_n^b(Z)$ and $R_n(Z)=R_n^b(Z)$, we first need an analogue of
Definition~\ref{def:9.3}.

\begin{definition}\label{def:8.10}
{\rm For an integer base $b\geq 2$, let $\beta$ be any $b$-ary overpartition, 
restricted or not, of some positive integer. We define $S^b(\beta)$ to be the 
sum of the numbers of
\begin{equation}\label{8.19}
\begin{cases}
\hbox{the different and single overlined parts of $\beta$, and}\\
\hbox{the different and single non-overlined parts of $\beta$, and}\\
\hbox{the different $(b-1)$-tuples of non-overlined parts of $\beta$, and}\\
\hbox{the different $b$-tuples of non-overlined parts of $\beta$.}
\end{cases}
\end{equation}}
\end{definition}

We can now state the $b$-ary analogue of Proposition~\ref{prop:9.4}.

\begin{proposition}\label{prop:8.11}
For $b\geq 2$, $n\geq 1$, and $0\leq j\leq n$, the coefficient $a_j^{(n)}$, as 
defined in \eqref{9.7}, counts the number of $b$-restricted $b$-ary 
overpartitions $\beta$ of $(b^{n+1}-b)/(b-1)$ with $S^b(\beta)=n+j$.
\end{proposition}

The proof of this follows from Proposition~\ref{prop:8.1}, 
Definition~\ref{def:8.10}, and \eqref{8.7}--\eqref{8.9}. One could also state
and prove $b$-ary analogues of Proposition~\ref{prop:9.6} and 
Corollaries~\ref{cor:9.5} and~\ref{cor:9.7}. We leave this to the reader.

\begin{example}\label{ex:9.4}
{\rm We take again $n=2$ and consider the thirteen 5-restricted 5-ary 
overpartitions of $(5^3-5)/(5-1)=30$, as displayed in Example~\ref{ex:9.2}.
Here we only list three representative examples, along with the sums 
$S^5(\beta)$ related to \eqref{8.19}:
\begin{align*}
&(5^5,1^5): 0+0+0+2=2;\qquad (5^5,\overline{1}, 1^4): 1+0+1+1=3;\\
&(\overline{5},5^4,\overline{1},1^4): 2+0+2+0=4.
\end{align*}
Altogether we have 8, 4, and 1 such 5-ary partitions with $S^5(\beta)=2, 3$,
and 4, respectively. This is consistent with Proposition~\ref{prop:8.11} and
is analogous to Example~\ref{ex:5.1}. }
\end{example}

\subsection{A third special case}

Finally, we set $x=1$ and $y_{b-1}=y_1$. Then an easy variant
of the proof of Proposition~\ref{prop:5.1} shows that in this case we have
\begin{equation}\label{8.20}
Q_n^b(Z)=\frac{1}{y_b+1}\left((y_1+y_b+1)^{n+1}-y_1^{n+1}\right),
\end{equation}
so that by \eqref{5.2} we have $Q_n^b(Z)=Q_n(1,y_1,y_b)$. Therefore this case
is covered by Sections~\ref{sec:5} and~\ref{sec:6}.

\section*{Acknowledgments}

We thank Keith Johnson of Dalhousie University for some helpful 
discussions.

\end{document}